\documentclass[12pt,a4paper]{article}
\usepackage[utf8]{inputenc}
\usepackage[english]{babel}
\usepackage{lmodern}
\usepackage[usenames, dvipsnames]{xcolor}
\usepackage{amssymb}
\usepackage{amsfonts}
\usepackage{amsmath}
\usepackage{amsthm}
\usepackage{mathtools}
\usepackage{dsfont}

\newtheorem{theorem}{Theorem}

\newtheorem{corollary}[theorem]{Corollary}
\theoremstyle{definition}

\allowdisplaybreaks[4]

\newcommand{\bN}{\mathbb{N}}

\newcommand{\bZ}{\mathbb{Z}}
\newcommand{\bR}{\mathbb{R}}
\newcommand{\bC}{\mathbb{C}}

\newcommand{\ii}{\textnormal{i}}

\DeclareMathOperator{\dz}{d\textit{z} \hspace{1pt}}

\begin{document}

\title{On a denseness result for quasi-infinitely divisible distributions \thanks{
This research was supported by DFG
grant LI-1026/6-1. Financial support is gratefully acknowledged.}}
\author{ Merve Kutlu \thanks{Ulm University, Institute of Mathematical Finance, D-89081 Ulm, Germany;
email: merve.kutlu@uni-ulm.de}}
\date{\today}
\maketitle


\begin{abstract}
A probability distribution $\mu$ on $\bR^d$ is quasi-infinitely divisible if its characteristic function has the representation $\widehat{\mu} = \widehat{\mu_1}/\widehat{\mu_2}$ with infinitely divisible distributions $\mu_1$ and $\mu_2$.
In \cite[Thm. 4.1]{lindner2018} it was shown that the class of quasi-infinitely divisible distributions on $\bR$ is dense in the class of distributions on $\bR$ with respect to weak convergence. 
In this paper, we show that the class of quasi-infinitely divisible distributions on $\bR^d$ is not dense in the class of distributions on $\bR^d$ with respect to weak convergence if $d \geq 2$.
\end{abstract}


Recently, there has been an increased interest in quasi-infinitely divisible distributions, see \cite{alexeev2021, berger2020, berger2021, khartov2019, lindner2018, nakamura2013, passeggeri2019, passeggeri2020}.
By definition, a distribution $\mu$ on $\bR^d$ is quasi-infinitely divisible if its characteristic function can be written as $\widehat{\mu} = \widehat{\mu_1}/\widehat{\mu_2}$ with infinitely divisible distributions $\mu_1$ and $\mu_2$.
The first systematic study of quasi-infinitely divisible distributions was made by Lindner et al. \cite{lindner2018} in one dimension, where it was shown among other things, that the class of quasi-infinitely divisible distributions on $\bR$ is dense in the class of all distributions on $\bR$ with respect to weak convergence. 
Passeggeri \cite{passeggeri2020} stated the conjecture that for general $d \in \bN$, the class of quasi-infinitely divisible distributions on $\bR^d$ is dense in the class of distributions on $\bR^d$ with respect to weak convergence, which was later formulated as an open question by Berger et al. \cite[Open Question 5.4]{berger2021}. 
By L\'evy's continuity theorem, this denseness is equivalent to the fact that the characteristic function of every distribution on $\bR^d$ can be written as the local uniform limit of characteristic functions of quasi-infinitely divisible distributions. 
In \cite[p. 22]{berger2021} it was also mentioned that it was shown by Pears in \cite[Prop. 3.3.2]{pears1975} that the set of all zero-free continuous complex-valued functions on $[0, 1]^d$ is dense in the set of all continuous complex-valued functions on $[0, 1]^d$ with respect to uniform convergence if and only if $d = 1$, which gives an indication that the denseness result conjectured by Passeggeri might not hold for $d > 1$. 
The present paper provides an answer to this question and shows that the class of quasi-infinitely divisible distributions on $\bR^d$ is dense in the class of distributions on $\bR^d$ if and only if $d=1$.
In particular, we give an example of a distribution on $\bR^d$, whose characteristic function can not be approximated arbitrarily well by continuous complex-valued functions with respect to local uniform convergence, and hence especially not by characteristic functions of quasi-infinitely divisible distributions, since those are zero-free.

Consider the distribution $\mu \coloneqq \frac{1}{3} \delta_{(0,1)} + \frac{1}{3} \delta_{(1,0)} + \frac{1}{3} \delta_{(1,1)}$ on $\bR^2$, where $\delta_a$ denotes the Dirac measure at a point $a \in \bR^2$. 
Its characteristic function $\varphi \coloneqq \widehat{\mu}: \bR^2 \to \bC$, $t \mapsto \int_{\bR^2} e^{\ii \langle t, z \rangle }\mu(\dz)$ is given by 
\begin{align} \label{phi}
	\varphi(x,y) 
	&= \frac{1}{3} \left( 
	e^{\ii x} + e^{\ii y} + e^{\ii (x+y)} \right), 
	\quad (x,y) \in \bR^2.
\end{align}
In Theorem \ref{bsp}, we show that $\varphi$ can not be approximated arbitrarily well by continuous functions without zeros with respect to uniform convergence on $[-\pi, \pi]^2$, where we present two different proofs. 
The first proof is a quite elementary topological proof, while the second proof constitutes an application of the Poincar\'e-Miranda theorem, which is equivalent to Brouwer's fixed-point theorem and can be seen as an extension of the mean value theorem to functions in several variables. 
For a proof of the Poincar\'e-Miranda theorem, see \cite[p. 457]{kulpa1997}.
For the first proof, we need the following topological terminology. 

Given a set $S \subset \bR^2$, a \emph{component} of $S$ is a maximal connected subset of $S$. 
We say that two points $x, y \in S$ are \emph{connected} in $S$ if a connected subset of $S$ contains $x$ and $y$.
This is the case, if and only if a component of $S$ contains $x$ and $y$, see \cite[p. 77]{newman1951}.
The set $S$ \emph{separates} two points $x, y \in \bR^2$ if $x$ and $y$ are not connected in $\bR^2 \setminus S$.  
A closed set $E \subset \bR^2$ is said to \emph{separate} the non-empty closed sets $A, B \subset \bR^2$ if there exist two disjoint, open sets $U$ and $V$ such that $A \subset U$, $B \subset V$ and $\bR^2 \setminus E = U \cup V$. 
Note that if $E$ separates the singletons $\{x\}$ and $\{y\}$ and $U,V \subset \bR^2$ are disjoint open sets such that $x \in U$, $y \in V$ and $\bR^2 \setminus E = U \cup V$, then every subset $F \subset \bR^2 \setminus E$ such that $x,y \in F$ can be written as the disjoint union of the non-empty sets $F \cap U$ and $F \cap V$ which are relatively open in $F$, so that $F$ is not connected.
Hence, in this case, $E$ also separates $x$ and $y$.

\begin{theorem} \label{bsp}
The characteristic function $\varphi$, given in \eqref{phi}, can not be approximated arbitrarily well by zero-free continuous functions from $\bR^2$ to $\bC$ with respect to local uniform convergence. 
\end{theorem}
{\it First proof.} 
Consider the sets 
\begin{align*}
&A_1 \coloneqq \left\{ 
	(x,y) \in [-\pi, \pi]^2: \Re \varphi(x,y) \geq \frac{1}{10}
\right\}, 
&& A_2 \coloneqq \left\{ 
	(x,y) \in [-\pi, \pi]^2: \Im \varphi(x,y) \geq \frac{1}{10}
\right\}, \\
&B_1 \coloneqq \left\{ 
	(x,y) \in [-\pi, \pi]^2: \Re \varphi(x,y) \leq -\frac{1}{10}
\right\},
&&B_2 \coloneqq \left\{ 
	(x,y) \in [-\pi, \pi]^2: \Im \varphi(x,y) \leq -\frac{1}{10}
\right\}.
\end{align*}
Obviously, $A_j$ and $B_j$ are closed and disjoint for $j \in \{1,2\}$, and $(0,0) \in A_1$, $(\pi, -\pi) \in B_1$, $(\pi/2,\pi/2) \in A_2$ and $(-\pi/2,-\pi/2) \in B_2$, so that they are also non-empty. 

Let $g: \bR^2 \to \bC$ be a continuous function such that $|g(x,y)-f(x,y)| < \frac{1}{20}$ for all $(x,y) \in [-\pi, \pi]^2$ and assume that $g$ has no zeros. 
Define the sets 
\begin{align*} 
&U \coloneqq \left\{
	(x,y) \in [-\pi, \pi]^2: \Re g(x,y) > 0
\right\}, \quad\\
&V \coloneqq \left\{
	(x,y) \in [-\pi, \pi]^2: \Re g(x,y) < 0
\right\} \cup (\bR^2 \setminus [-\pi, \pi]^2), \\
&E \coloneqq \left\{
	(x,y) \in [-\pi, \pi]^2: \Re g(x,y) = 0
\right\}. 
\end{align*}
It can  be seen that $U$ and $V$ are disjoint open subsets of $\bR^2$, since $\partial (\bR^2 \setminus [-\pi, \pi]^2) \subset B_1 \subset V$. 
Indeed, we have
\begin{align*}
\varphi(\pi,y) 
= \frac{1}{3} \left( 
	e^{\ii \pi} + e^{\ii y} + e^{\ii \pi} e^{\ii y} 
\right)
= -\frac{1}{3} < -\frac{1}{10}
\end{align*}
for $y \in [-\pi, \pi]$. 
Similarly, it holds $\varphi(-\pi, y) = \varphi(x, \pm \pi) = -\frac{1}{3}$ for all $x,y \in [-\pi, \pi]$.
Moreover, $E$ is a compact set such that $\bR^2 \setminus E = U \cup V$. 
Since further $ A_1 \subset U$ and $B_1 \subset V$, $E$ separates $A_1$ and $B_1$. 
It is clear that $E \cap A_2 $ and $E \cap B_2$ are closed sets. 
Suppose that there exists a component $F$ of $E$ such that $F \cap A_2 \neq \emptyset$ and $F \cap B_2 \neq \emptyset$. 
If $(x_1, y_1) \in F \cap A_2$ and $(x_2, y_2) \in F \cap B_2$, then $\Im g(x_1, y_1) > 0$ and $\Im g(x_2, y_2) < 0$. 
Since $F$ is connected and $\Im g$ is continuous, also $\Im g(F)$ is connected, so that there exists $(x_0, y_0) \in F$ with $\Im g(x_0, y_0) = 0$, and hence $g(x_0, y_0) = 0$, a contradiction.
Hence, there exists no component of $E$ which intersects both $E \cap A_2$ and $E \cap B_2$. 
By \cite[Thm. 5.6]{newman1951} there exist two disjoint, closed sets $E_1$ and $E_2$ such that $E = E_1 \cup E_2$ as well as $E_1 \cap B_2 = \emptyset$ and $E_2 \cap A_2 = \emptyset$. 

Now consider the points $(0,0) \in A_1$ and $(\pi, -\pi) \in B_1$, which are separated by $E$. 
By \cite[Thm. 14.3]{newman1951}, there exists one component $E_0$ of $E$ which separates $(0,0)$ and $(\pi, -\pi)$.
Then either $E_0 \subset E_1$ or $E_0 \subset E_2$.
If $E_0 \subset E_1$, then especially, $E_0 \subset [-\pi, \pi]^2 \setminus (A_1 \cup B_1 \cup B_2)$. 
In this case we consider the path $\gamma: [0,1] \to \bR^2$ given by 
\begin{align*}
\gamma(t) = \begin{cases}
(0, -2\pi t) \quad & \textrm{if } t \in [0, 1/2),\\
(2\pi t - \pi , -\pi) \quad & \textrm{if } t \in [1/2, 1].\\
\end{cases}
\end{align*}

The path $\gamma$ connects $(0,0)$ and $(\pi, -\pi)$, and $\gamma(t) \in A_1 \cup B_1 \cup B_2$ for all $t \in [0,1]$. 
Indeed, for $t \in [0, 1/4]$ we estimate 
\begin{align*}
\Re \varphi(\gamma(t))
= 2/3 \cos(2\pi t) + 1/3 
\geq 1/3 > 1/10,
\end{align*}
so $\gamma(t) \in A_1$, and for $t \in [1/4, 3/8]$ it holds 
\begin{align*}
\Im \varphi(\gamma(t))
= -2/3 \sin(2\pi t) 
\leq -1/3 < -1/10,
\end{align*}
and hence $\gamma(t) \in B_2$. 
If $t \in [3/8, 1/2]$, then we have 
\begin{align*}
\Re \varphi(\gamma(t))
= 2/3 \cos(2\pi t) + 1/3 
\leq 2/3(-1/\sqrt{2})+1/3 < - 1/10,
\end{align*}
so that $\gamma(t) \in B_1$ and for $t \in [1/2, 1]$,
\begin{align*}
 \varphi(\gamma(t))
&= -1/3 < -1/10,
\end{align*}
as seen above, showing that $\gamma(t) \in B_1$.
Therefore, $\gamma$ does not intersect $[-\pi, \pi]^2 \setminus (A_1 \cup B_1 \cup B_2)$, and hence also not $E_0$.
Hence, $\Gamma \coloneqq \gamma([0,1])$ is a connected subset of $ \bR^2 \setminus E_0$ which contains $(0,0)$ and $(\pi, -\pi)$, showing that  $(0,0)$ and $(\pi, -\pi)$ are connected in $\bR^2 \setminus E_0$. 
But this contradicts the fact that $E_0$ separates $(0,0)$ and $(\pi, -\pi)$.

If $E_0 \subset E_2$, then $E_0 \subset [-\pi, \pi]^2 \setminus (A_1 \cup B_1 \cup A_2)$ and a similar contradiction can be obtained by considering the path 
\begin{align*}
\gamma: [0,1] \to \bR^2, \quad
\gamma(t) = \begin{cases}
(2\pi t, 0) \quad & \textrm{if } t \in [0, 1/2),\\
(\pi , \pi - 2\pi t ) \quad & \textrm{if } t \in [1/2, 1].\\
\end{cases}
\end{align*}

{\it Second proof. }
Consider the affine transformation $\psi: [0,1]^2 \to \bR^2$ of the unit square, defined by 
\begin{align*}
\psi(x,y) 
= \left(\frac{5\pi}{8}, -\frac{7\pi}{8}\right)
	+ x \left(\frac{\pi}{4}, \frac{\pi}{4}\right)
	+ y \left(-\frac{\pi}{2}, \frac{\pi}{2}\right),
\quad (x,y) \in [0,1]^2.
\end{align*}
Then
\begin{align*}
&\Re \varphi(\psi(x,0))
= \frac{1}{3} \cos\left(\frac{\pi}{8}(5+2x)\right)
	+ \frac{1}{3} \cos\left(\frac{\pi}{8}(-7+2x)\right)
	+ \frac{1}{3} \cos\left(\frac{\pi}{8}(-2+4x)\right),\\
&\Re \varphi(\psi(x,1))
= \frac{1}{3} \cos\left(\frac{\pi}{8}(1+2x)\right)
	+ \frac{1}{3} \cos\left(\frac{\pi}{8}(-3+2x)\right)
	+ \frac{1}{3} \cos\left(\frac{\pi}{8}(-2+4x)\right),\\
&\Im \varphi(\psi(0,y))
= \frac{1}{3} \sin\left(\frac{\pi}{8}(5-4y)\right)
	+ \frac{1}{3} \sin\left(\frac{\pi}{8}(-7+4y)\right)
	-\frac{1}{3\sqrt{2}}
\intertext{and}
&\Im \varphi(\psi(1,y))
= \frac{1}{3} \sin\left(\frac{\pi}{8}(7-4y)\right)
	+ \frac{1}{3} \sin\left(\frac{\pi}{8}(-5+4y)\right)
	+ \frac{1}{3\sqrt{2}}
\end{align*}
for $x, y \in [0,1]$. An easy curve discussion shows that
\begin{align*}
\Re \varphi(\psi(x,y))
\begin{cases}
	< -0.05  \quad& \textrm{if } y = 0\\
	> 0.05 \quad& \textrm{if } y = 1
\end{cases}
\quad \textrm{and} \quad
\Im \varphi(\psi(x,y))
\begin{cases}
	< -0.05  \quad& \textrm{if }x = 0\\
	> 0.05 \quad& \textrm{if } x = 1.
\end{cases}
\end{align*}
Suppose that $g: \bR^2 \to \bC$ is a continuous function such that $|g(x,y)-\varphi(x,y)| < 0.025$ for all $(x,y) \in \psi([0,1]^2)$. 
Then especially $|\Re g(x,y)-\Re \varphi(x,y)|, |\Im g(x,y)-\Im \varphi(x,y)| < 0.025$ for  $(x,y) \in \psi([0,1]^2)$, and hence
\begin{align*}
\Re g(\psi(x,y))
\begin{cases}
	< -0.025  \quad& \textrm{if } y = 0\\
	> 0.025 \quad& \textrm{if } y = 1
\end{cases}
\quad \textrm{and} \quad
\Im g(\psi(x,y))
\begin{cases}
	< -0.025  \quad& \textrm{if }x = 0\\
	> 0.025 \quad& \textrm{if } x = 1.
\end{cases}
\end{align*}
By the Poincar\'e-Miranda Theorem, it follows that $g \circ \psi$ has a zero in $[0,1]^2$, which corresponds to a zero of $g$ in $\bR^2$.
\qed

\pagebreak
\begin{corollary}
For $d \in \bN$ the following statements are equivalent. 
\begin{enumerate}
\item The class of quasi-infinitely divisible distributions on $\bR^d$ is dense in the class of all distributions on $\bR^d$ with respect to weak convergence. 
\item Every characteristic function on $\bR^d$ can be written as the local uniform limit of zero-free continuous complex-valued functions. 
\item $d = 1$. 
\end{enumerate}
\end{corollary}
\begin{proof}
That (i) implies (ii) is due to L\'evy's continuity theorem and the fact that the characteristic function of a quasi-infinitely divisible distribution can not have zeros. 

We have already seen in Theorem \ref{bsp} that the characteristic function of the distribution $\mu \coloneqq \frac{1}{3} \delta_{(0,1)} + \frac{1}{3} \delta_{(1,0)} + \frac{1}{3} \delta_{(1,1)}$ on $\bR^2$ can not be approximated arbitrarily well by continuous functions without zeros with respect to local uniform convergence. 
Now suppose that $d>2$ and consider the distribution 
\begin{align*}
\sigma \coloneqq \mu \otimes \delta_0^{\otimes (d-2)}
\end{align*}
on $\bR^d$.
The characteristic function of $\sigma$ is given by 
\begin{align*}
\widehat{\sigma}(x_1, \ldots, x_d)
= \widehat{\mu}(x_1, x_2),
\quad (x_1, \ldots, x_d) \in \bR^d.
\end{align*}
Assume that there exists a sequence $(f_n)_{n \in \bN}$ of zero-free, continuous complex-valued functions on $\bR^d$ such that $f_n \to \widehat{\sigma}$ as $n \to \infty$ locally uniformly on $\bR^d$. 
For $n \in \bN$ let $g_n: \bR^2 \to \bC, (x_1, x_2) \mapsto f(x_1, x_2, 0, \ldots, 0)$.
Then for every compact set $K_0 \subset \bR^2$ we have that $f_n \to \widehat{\sigma}$ uniformly on $K_0 \times \{0\}^{d-2}$ as $n \to \infty$, and hence $g_n \to \widehat{\sigma}$ uniformly on $K_0$ as $n \to \infty$.
But this contradicts the fact that $\widehat{\mu}$ can not be approximated arbitrarily well with respect to local uniform convergence. 
Hence, (iii) follows from (ii).

Finally, Lindner et al. showed in \cite[Thm. 4.1]{lindner2018} that the class of quasi-infinitely divisible distributions on $\bR$ is dense in the class of distributions on $\bR$, so that (iii) implies (i).
\end{proof}

\subsection*{Acknowledgements}
The author would like to thank David Berger, Anna Dall'Acqua and Alexander Lindner for many fruitful discussions and for pointing out several references.

\end{document}